\documentclass[oneside,a4paper,12pt]{article}

\usepackage{geometry} 
\usepackage{amsmath}
\usepackage{amsthm}
\usepackage{amsfonts,amssymb}
\usepackage{fancyhdr} 
\usepackage{mathrsfs}  
\usepackage{graphicx}
\usepackage[all]{xy}
\usepackage{comment}
\usepackage{cases}
\usepackage{color}
\usepackage{placeins}
\usepackage[colorlinks,linkcolor=blue,anchorcolor=blue,citecolor=blue]{hyperref}
\usepackage{xcolor}
\usepackage{graphicx}
\usepackage{tikz}

\usepackage{float} 
\usepackage{subfig}

\newcommand{\R}{\mathbb R}
\newcommand{\T}{\mathcal T}

\renewcommand{\H}{\mathbb{H}}

\newcommand{\Aff}{\operatorname{Aff}}
\newcommand{\Reg}{\operatorname{Reg}}
\newcommand{\conv}{\operatorname{conv}}
\newcommand{\Vol}{\operatorname{Vol}}

\theoremstyle{plain}

\newtheorem{theorem}{Theorem}[section]
\newtheorem{lemma}[theorem]{Lemma}
\newtheorem{proposition}[theorem]{Proposition}
\newtheorem{corollary}[theorem]{Corollary}

\theoremstyle{definition}
\newtheorem{definition}[theorem]{Definition}

\theoremstyle{remark}
\newtheorem{remark}[theorem]{Remark}
\newtheorem{example}[theorem]{Example}

\numberwithin{equation}{section}

\begin{document}

\title{\textbf{Geometric ideal triangulations of hyperbolic $3$-manifolds}}
\author{Huabin Ge}
\date{}
\maketitle

\begin{abstract}
We resolve the longstanding geometric ideal triangulation conjecture arising from Thurston's theory of hyperbolic $3$-manifolds: every finite-volume cusped hyperbolic $3$-manifold admits a finite geometric ideal triangulation realizing its complete hyperbolic structure.
\end{abstract}

\section{Introduction}\label{section:introduction}
A geometric ideal triangulation of a hyperbolic 3-manifold $M$ is an expression of $M$ as a union of geometric ideal tetrahedra glued along their faces. 
Whether every cusped hyperbolic $3$-manifold admits a geometric ideal triangulation is a longstanding problem in $3$-manifold topology, listed as Problem 3.3
in the 2026 Kirby problem list \cite{KirbyK3}. We resolve this problem affirmatively.

\begin{theorem}\label{main}
Every complete finite-volume cusped hyperbolic $3$-manifold $M$ admits a finite geometric ideal triangulation by hyperbolic ideal tetrahedra realizing its complete hyperbolic structure.
\end{theorem}

When $M$ is orientable, one may choose vertex orderings consistent with the global orientation so that all shape parameters lie in the upper half-plane. We prove the main theorem starting with the classical Epstein-Penner decomposition \cite{EP}, that is, after a choice of cusp decoration, every such manifold admits an Epstein--Penner decomposition into finitely many ideal hyperbolic polyhedra. Together with the following subdivision compatibility theorem, we obtain the desired conclusion.

\begin{theorem}[Compatible tetrahedralization]
\label{main-2}
Let $\mathcal P$ be a finite system of three-dimensional convex polyhedra whose two-dimensional faces are paired by affine homeomorphisms. Then the polyhedra admit tetrahedral subdivisions, using only their original vertices, such that the subdivisions agree under every face pairing.
\end{theorem}

For manifolds with totally geodesic boundary, we use the polyhedral decompositions of Kojima \cite{Kojima} and Ge-Jia-Zhang \cite{GJZdecomp}, together with
the affine-fellow correspondence of Luo-Schleimer-Tillmann \cite{LST}. This yields two further results.

\begin{theorem}\label{main:boundary}
Let $N$ be an orientable, complete, finite-volume hyperbolic $3$-manifold with nonempty compact totally geodesic boundary and possibly toric cusps. Then $N$ admits a finite good geometric partially truncated triangulation.
\end{theorem}

Here \emph{good} means that each tetrahedron
has at most one ideal vertex. In particular, this proves Frigerio's Conjecture $1.2$ in \cite{Frigerio}.

\begin{theorem}\label{main:mixed}
Let $M$ be an orientable, finite-volume, noncompact, complete hyperbolic $3$-manifold with totally geodesic boundary in the setting of \cite{GJZdecomp}. Then $M$ admits a finite geometric partially truncated triangulation such that every tetrahedron is either ideal or has exactly one hyperideal vertex. No flat tetrahedron is introduced, and the triangulation realizes the given complete hyperbolic structure.
\end{theorem}

\subsection*{Background and previous work}
The existence of geometric ideal triangulations is a longstanding problem arising from Thurston's theory of hyperbolic $3$-manifolds \cite{Thurston}. Thurston's original proof of
hyperbolic Dehn filling implicitly assumed their existence, an assumption later removed by Petronio-Porti \cite{PP}. For this reason, the existence problem is sometimes referred to as Thurston's geometric ideal triangulation conjecture, or simply Thurston's triangulation conjecture.

Epstein-Penner \cite{EP} showed that every cusped hyperbolic 3-manifold admits a decomposition into convex ideal polyhedra. Each such polyhedron can be subdivided into
ideal tetrahedra. The difficulty is to make the subdivisions compatible across paired faces; see \cite{KirbyK3,FHH}. Since then, several important partial results have been obtained. Yoshida \cite{Yoshida} proved the conjecture for manifolds obtained from two convex ideal polyhedra with all faces paired between them. Wada-Yamashita-Yoshida \cite{WYY} and Sirotkina \cite{Sirotkina} treated further special configurations. Petronio-Weeks \cite{PW} showed that the Epstein-Penner decomposition can be subdivided into ideal tetrahedra if flat tetrahedra are allowed to resolve incompatible face triangulations. See Petronio \cite{Petronio} for a survey of partially flat ideal triangulations and related existence problems, and see \cite{Gue1,Lackenby,Nim} for some explicit geometric ideal triangulations of special classes of hyperbolic 3-manifolds.

Luo-Schleimer-Tillmann \cite{LST} proved that every cusped hyperbolic $3$-manifold admits a finite cover with a geometric ideal triangulation. Futer-Hamilton-Hoffman \cite{FHH} strengthened this virtual existence result by constructing infinitely many geometric ideal triangulations in a suitable finite cover.

A complementary approach uses combinatorial Ricci flow. Costantino-Frigerio-Martelli-Petronio \cite{CFMP} proved hyperbolicity under the edge-valence bound $6$ and conjectured that the prescribed triangulation can be geometrically realized by hyperbolic partially truncated tetrahedra. Using combinatorial Ricci flow, Feng-Ge-Hua \cite{FGH} proved this under the stronger bound $10$, thereby establishing both hyperbolization and a geometric triangulation of the prescribed ideal triangulation.
Zhao \cite{Zhao} improved the bound to $9$. More generally, Feng-Ge \cite{FengGe} showed, in their setting, that the extended combinatorial Ricci flow converges if and only if the prescribed triangulation is geometric. Related results on combinatorial Ricci flows can be found in \cite{FGHX,FGM}.

Other developments include the existence of ideal triangulations carrying strict angle structures under a homological hypothesis,
established by Hodgson-Rubinstein-Segerman \cite{HRS} (see also \cite{GJZangle} for a related generalization), and explicit geometric triangulations for important families of hyperbolic link
complements constructed by Ham-Purcell \cite{HP}. Choi \cite{Choi} also exhibited an incomplete hyperbolic structure that does not admit a
geometric ideal triangulation, highlighting the importance of the completeness assumption.

\subsection*{Organization of the paper}
Section \ref{section:preliminaries} recalls the Klein model, the Epstein-Penner decomposition and regular subdivisions. Section \ref{section:compatible} proves the compatible triangulation theorem for affine face pairings. Section \ref{section:proof-main} applies it to the Epstein-Penner decomposition. Section \ref{section:projective} treats projective face pairings and manifolds with geodesic boundary. Section \ref{section:consequences} discusses the applications above.

~

\noindent
\textbf{Acknowledgements:}
The author is grateful to Longsong Jia and Rongyuan Du for helpful discussions, preparing the figures and careful checking of details. This work is supported by NSFC under Grants Nos. 12341102, 12122119, and 12525103. The author used ChatGPT (OpenAI) to assist with English language editing. All mathematical ideas, results, arguments, and proofs are the author's own. The authors have carefully reviewed the entire manuscript and assume full intellectual responsibility for its content.


\section{Hyperbolic polyhedra and regular subdivisions}\label{section:preliminaries}

In this section, we recall the Klein model, the Epstein-Penner decomposition and the facts about regular subdivisions used later.

\subsection{The hyperboloid and Klein models}
Equip $\R^{3,1}$ with the Lorentzian form $\langle x,y\rangle=-x_0y_0+x_1y_1+x_2y_2+x_3y_3$. Let
\begin{equation}
\H^3=\{x\in\R^{3,1}:\langle x,x\rangle=-1,\ x_0>0\}
  \nonumber \end{equation}
and let $L^+$ denote the future light cone. A nonzero vector is timelike if $\langle x,x\rangle<0$, lightlike if $\langle x,x\rangle=0$, and future-directed if $x_0>0$.

Radial projection to the affine chart $x_0=1$ gives the Klein model,
\begin{equation}\label{equation:klein-projection-prelim}
\pi(x)=\left(\frac{x_1}{x_0},\frac{x_2}{x_0},\frac{x_3}{x_0}\right).
\end{equation}
In this model geodesics are Euclidean line segments and totally geodesic planes are intersections of the Klein ball with affine planes. If $v_1,\ldots,v_r$ have positive $0$-th coordinate and $t_i\geq0$, then
\begin{equation}\label{equation:radial-convexity-prelim}
\pi\!\left(\sum_i t_iv_i\right)
 =\sum_i\frac{t_i(v_i)_0}{\sum_jt_j(v_j)_0}\,\pi(v_i).
\end{equation}
Thus radial projection preserves convex hulls. We refer to \cite{BenedettiPetronio} for these models.

\subsection{The Epstein-Penner decomposition}\label{subsection:EP}
Let $M=\H^3/\Gamma$ be a complete finite-volume cusped hyperbolic $3$-manifold. We do not assume that $M$ is orientable. Thus
$\Gamma$ is a torsion-free discrete subgroup of
$O^+(3,1)$, not necessarily of $SO^+(3,1)$. Choose a horospherical decoration of the cusps and let $\mathcal V\subset L^+$ be the union of the corresponding $\Gamma$-orbits of light vectors. Set
\begin{equation}
C=\operatorname{conv}(\mathcal V)\subset\R^{3,1}.
  \nonumber \end{equation}
Epstein--Penner \cite{EP} show that the three-dimensional faces of $\partial C$, after radial projection to the Klein model, give a locally finite $\Gamma$-invariant decomposition of $\H^3$ into convex ideal polyhedra. The quotient contains finitely many cells; see also \cite{Akiyoshi}.

Let $P$ be a three-dimensional face of $\partial C$ and set $A=\Aff(P)$. For the faces in the Epstein-Penner construction, $A$ is a spacelike supporting hyperplane of $C$ \cite{EP}. Hence, after normalization,
\begin{equation}\label{equation:support-plane-prelim}
A=\{x\in\R^{3,1}:\lambda(x)=1\},\qquad
\lambda(x)=-\langle x,w\rangle,
\end{equation}
for a future-directed timelike vector $w$. The vertices of $P$ lie in $L^+$. Since $A\cap L^+$ is strictly convex, the light vectors spanning $P$ are vertices of $P$.

Choose one lift
\begin{equation}
P_i\subset A_i,\qquad i=1,\ldots,N,
  \nonumber \end{equation}
for each three-dimensional cell of the quotient decomposition. Let $F$ be a two-dimensional face of $P_i$, and let $Q$ be the lifted cell on the other side of $F$. There are $j$ and $\gamma\in\Gamma$ such that $\gamma(Q)=P_j$. Then $\gamma(F)$ is a face of $P_j$, and
\begin{equation}
\gamma|_F:F\longrightarrow\gamma(F)
  \nonumber \end{equation}
is an affine homeomorphism. The reverse pairing is induced by $\gamma^{-1}$. Indeed, the deck transformation $\gamma$ acts
linearly on $\mathbb R^{3,1}$ and maps the affine
span of $Q$ onto $A_j$. Consequently, its
restriction to the paired face is affine.

No nontrivial element of $\Gamma$ stabilizes a lifted cell or one of its faces. Indeed, such an element permutes the vertices and fixes their barycenter. The barycenter is future-directed timelike and hence determines a point of $\H^3$, which is impossible for a nontrivial element of the torsion-free group $\Gamma$.

\subsection{Regular subdivisions}
Let $A$ be a real affine space, let $P\subset A$ be a convex polyhedron with vertex set $V$, and let $h:V\to\R$. Lift each vertex to
\begin{equation}
\widehat v=(v,h(v))\in A\times\R
  \nonumber
\end{equation}
and set
\begin{equation}
\widehat P_h=\operatorname{conv}\{\widehat v:v\in V\}.
  \nonumber
\end{equation}
A face of $\widehat P_h$ is called a \emph{lower face} if it is supported by a hyperplane of the form
\begin{equation}
t=r(x),
  \nonumber
\end{equation}
where $r:A\to\R$ is affine and $\widehat P_h$ lies in the half-space $t\geq r(x)$. The projections of the lower faces of $\widehat P_h$ to $A$ form the \textbf{regular (or coherent) subdivision} determined by $h$, denoted by $\Reg(P,h)$.

Equivalently, the cells of $\Reg(P,h)$ are the nonempty convex hulls
\begin{equation}
\operatorname{conv}\{v\in V:r(v)=h(v)\},
  \nonumber
\end{equation}
where $r$ ranges over affine functions satisfying
\begin{equation}
r(v)\leq h(v)\qquad(v\in V).
  \nonumber
\end{equation}
Thus contact with a single lifted vertex, an edge, or a higher-dimensional lower face gives a cell of the corresponding dimension; empty contact sets contribute no cell. Adding the restriction of an affine function on $A$ to $h$ does not change the subdivision.

The following restriction property is standard; see \cite[Lemma 2.3.15]{DRS}.

\begin{lemma}\label{lemma:restriction}
If $F$ is a face of $P$, then
\begin{equation}
\Reg(P,h)|_F=\Reg(F,h|_{V(F)}).
  \nonumber \end{equation}
\end{lemma}

\begin{proof}
Let $C$ be a cell of $\Reg(P,h)$ with $C\cap F\neq\varnothing$. Choose an affine function $r:A\to\R$ such that $r(v)\leq h(v)$ for all $v\in V$ and the vertices of $C$ are precisely the vertices where equality holds. Since $F$ is a face of $P$, the vertices of $C\cap F$ are exactly the contact vertices lying in $F$. Hence $r|_{\Aff(F)}$ determines $C\cap F$ as a cell of $\Reg(F,h|_{V(F)})$.

Conversely, let $D$ be a cell of $\Reg(F,h|_{V(F)})$, determined by an affine function $r_F$ on $\Aff(F)$. Extend $r_F$ to an affine function $r_0$ on $A$. If $F=P$ there is nothing to prove. Otherwise, since $F$ is a face of $P$, there is an affine function $q:A\to\R$ which vanishes on $\Aff(F)$ and is positive at every vertex of $P$ outside $F$. For sufficiently large $K$, the affine function $r_0-Kq$ lies below $h$ at every vertex outside $F$, strictly so, while its values on $F$ are unchanged. Its contact vertices are therefore exactly those determining $D$. Thus $D$ is the intersection with $F$ of a cell of $\Reg(P,h)$, proving the equality of the induced subdivisions.
\end{proof}

\section{Admissible polyhedral systems}\label{section:compatible}

The compatibility problem is naturally stated in the following affine setting. In this section, a three-dimensional convex polyhedron in an affine $3$-space means the convex hull of finitely many points with nonempty interior, and a face means a two-dimensional face unless otherwise specified. We encode regular subdivisions by vertex
heights. Compatibility across paired faces is expressed by linear conditions modulo affine functions. Euler's formula then gives the dimension estimate in Lemma \ref{lemma:surplus}. Related compatible tetrahedralization problems, including the extension of prescribed boundary triangulations, were studied by Bern \cite{Bern}.

\begin{definition}\label{definition:admissible-system}
An \emph{admissible polyhedral system} consists of finitely many three-dimensional convex polyhedra
\begin{equation}
\mathcal P=\{P_1,\ldots,P_N\},\qquad P_i\subset A_i,
  \nonumber \end{equation}
where each $A_i$ is a real affine $3$-space, together with a pairing of all two-dimensional faces. If $F^-_\alpha\subset P_i$ is paired with $F^+_\alpha\subset P_j$, the pairing is an affine homeomorphism
\begin{equation}
\phi_\alpha:F^-_\alpha\longrightarrow F^+_\alpha,
  \nonumber \end{equation}
and the reverse pairing is $\phi_\alpha^{-1}$. Two faces of the same polyhedron may be paired.

A \emph{height function} on $\mathcal P$ is a collection $h=(h_i)$ with
\begin{equation}
h_i:V(P_i)\longrightarrow\R.
  \nonumber \end{equation}
It is \emph{compatible} if for every face pairing $\phi_\alpha:F^-_\alpha\to F^+_\alpha$,
\begin{equation}\label{equation:compatible-height}
h^-_\alpha\sim h^+_\alpha\circ\phi_\alpha.
\end{equation}
Two height functions $h=(h_i)$ and $h'=(h'_i)$ on $\mathcal P$ are equivalent, written $h\sim h'$, if $h_i-h_i'$ is the restriction to $V(P_i)$ of an affine function on $\Aff(P_i)$ for every $i$.
\end{definition}

Given a height function $h=(h_i)$, subdivide each polyhedron $P_i$ by $\Reg(P_i,h_i)$. If $h$ is compatible, these subdivisions agree on every pair of identified faces. Indeed, by Lemma \ref{lemma:restriction}, the subdivision induced on a face $F\subset P_i$ is $\Reg(F,h_i|_{V(F)})$; on paired faces the induced height functions are equivalent and therefore determine the same regular subdivision. Thus a compatible height function gives a compatible subdivision of the whole polyhedral system.

\begin{figure}[htbp]
\centering
\includegraphics[scale=0.35]{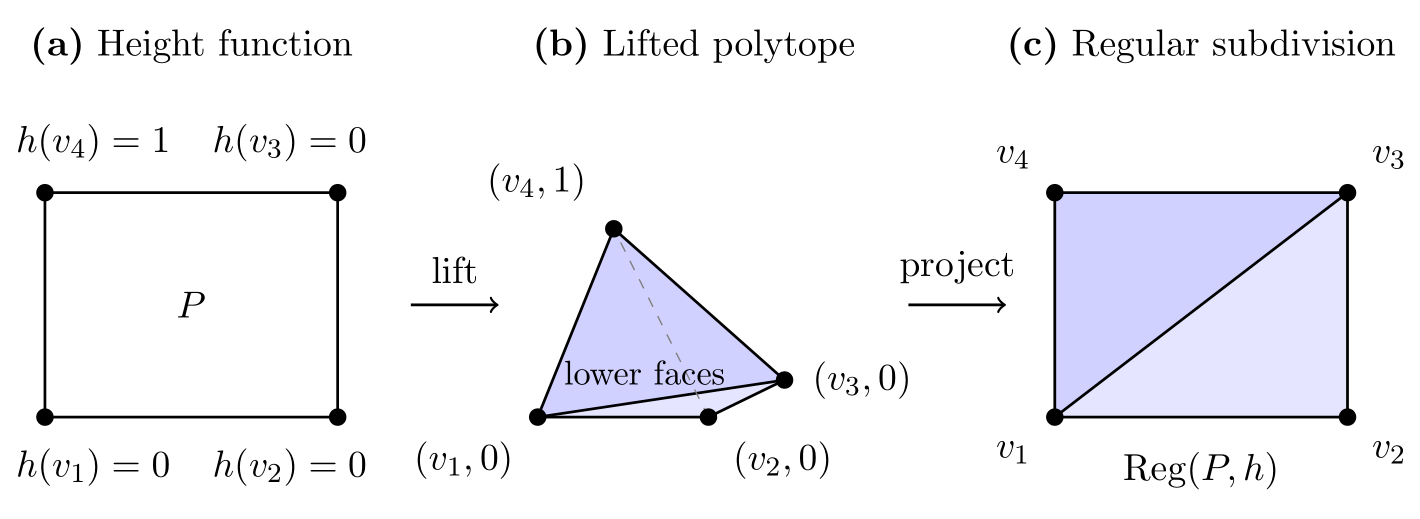}
\caption{
A height function and its induced regular subdivision.
Here $h(v_1)=h(v_2)=h(v_3)=0$ and $h(v_4)=1$.
The two lower faces project to
$\conv\{v_1,v_2,v_3\}$ and
$\conv\{v_1,v_3,v_4\}$,
so $\Reg(P,h)$ is the subdivision of $P$ along the diagonal
$v_1v_3$.
}
\label{fig:regular-subdivision}

\end{figure}

\begin{lemma}\label{lemma:surplus}
Let $f_i$ be the number of two-dimensional faces of $P_i$. The vector space of compatible height classes on $\mathcal P$ has dimension at least
\begin{equation}\label{equation:surplus}
\frac12\sum_i(f_i-4).
\end{equation}
In particular, if one of the $P_i$ is not a tetrahedron, there is a nonzero compatible height class.
\end{lemma}

\begin{proof}
Let $V_i=V(P_i)$ and set
\begin{equation}
\Aff(P_i)=\{r|_{V_i}:r:A_i\to\R\text{ affine}\},
\qquad
H_i=\R^{V_i}/\Aff(P_i).
  \nonumber \end{equation}
Since $A_i$ has dimension three,
\begin{equation}
\dim H_i=|V_i|-4.
  \nonumber \end{equation}
For a paired face $F^-_\alpha$ with $n_\alpha$ vertices, let
\begin{equation}
\Aff(F^-_\alpha)=\{r|_{V(F^-_\alpha)}:r:\operatorname{aff}(F^-_\alpha)\to\R\text{ affine}\}
  \nonumber \end{equation}
and put
\begin{equation}
Q_\alpha=\R^{V(F^-_\alpha)}/\Aff(F^-_\alpha),
\qquad
\dim Q_\alpha=n_\alpha-3.
  \nonumber \end{equation}
The face pairings define a linear map
\begin{equation}
B:\bigoplus_i H_i\longrightarrow\bigoplus_\alpha Q_\alpha,
\qquad
B_\alpha([h])=[h^-_\alpha-h^+_\alpha\circ\phi_\alpha].
  \nonumber \end{equation}
It is well defined because the restriction of an affine function remains affine after composition with an affine homeomorphism. The compatible height classes are precisely $\ker B$.

Write $v_i,e_i,f_i$ for the numbers of vertices, edges and two-dimensional faces of $P_i$, and let $F$ be the number of pairs of faces. Then
\begin{equation}
\sum_i f_i=2F,
\qquad
\sum_i e_i=\sum_\alpha n_\alpha.
  \nonumber \end{equation}
Euler's formula gives
\begin{equation}
\dim\bigoplus_iH_i-\dim\bigoplus_\alpha Q_\alpha
 =\sum_i(v_i-4)-\sum_\alpha(n_\alpha-3)
 =\frac12\sum_i(f_i-4).
  \nonumber \end{equation}
Hence $\dim\ker B$ is at least the number in \eqref{equation:surplus}. Since a three-dimensional convex polyhedron has at least four faces, with equality only for a tetrahedron, the last statement follows.
\end{proof}

\begin{proposition}\label{proposition:one-refinement}
Let $\mathcal P$ be an admissible polyhedral system containing a non-tetrahedral polyhedron. There is a compatible height function for which the regular subdivisions are not all trivial. The three-dimensional cells of these subdivisions, with the induced face pairings, form another admissible polyhedral system.
\end{proposition}

\begin{proof}
Choose a nonzero compatible height class by Lemma \ref{lemma:surplus}. At least one subdivision is nontrivial. Indeed, if $\Reg(P_i,h_i)$ were trivial for every $i$, then $P_i$ itself would be a lower cell for each $i$, so $h_i$ would agree with an affine function at every vertex of $P_i$. The height class would therefore be zero.

On a paired face, the induced subdivisions agree by compatibility and Lemma \ref{lemma:restriction}. If a new face $G$ lies in an original paired face $F$ with pairing $\phi:F\to F'$, then $\phi(G)$ is a face of the induced subdivision of $F'$, and $G$ is paired with $\phi(G)$ by $\phi|_G$. An interior face is shared by exactly two three-dimensional cells. Regarding these cells as disjoint, we pair the two copies of the face by the identity. Thus the three-dimensional cells, with these induced face pairings, form an admissible polyhedral system.
\end{proof}

\begin{theorem}\label{theorem:affine-triangulation}
Every admissible polyhedral system can be subdivided into tetrahedra, without adding vertices, such that the two triangulations of every pair of identified faces coincide under the given identification. The tetrahedra therefore glue together using the original face identifications.
\end{theorem}

\begin{proof}
Apply Proposition \ref{proposition:one-refinement} as long as a non-tetrahedral cell remains. Each application increases the number of three-dimensional cells.

Fix an original polyhedron $P_i$. Every three-dimensional cell contained in $P_i$ contains four affinely independent vertices of $P_i$. Two distinct cells in the same subdivision cannot contain the same four vertices, since the tetrahedron spanned by them has nonempty interior. Hence the number of cells contained in $P_i$ is bounded by the number of affinely independent four-element subsets of $V(P_i)$. The refinement process is therefore finite.

At the final stage every cell is a tetrahedron. The identity pairings on interior faces recover each $P_i$, while the boundary pairings are restrictions of the original ones. Hence these tetrahedra descend to a triangulation of the
original quotient in the usual face-pairing sense, allowing
vertex identifications.
\end{proof}

\begin{example}\label{example:cube}
Let $P=[0,1]^3$, with opposite faces paired by translations. Write
\begin{equation}
v_{ijk}=(i,j,k),\qquad i,j,k\in\{0,1\},
  \nonumber \end{equation}
and abbreviate $v_{ijk}$ to $ijk$. All eight vertices represent the same vertex in the quotient.

Up to addition of an affine function, we can normalize a height by
\begin{equation}
h(000)=h(100)=h(010)=h(001)=0,
  \nonumber \end{equation}
and set
\begin{equation}
h(110)=a,\qquad h(101)=b,\qquad h(011)=c,\qquad h(111)=d.
  \nonumber \end{equation}
On the four vertices $001,101,011,111$ of a square, a height function $k$ is affine exactly when
\begin{equation}
k_{001}+k_{111}=k_{101}+k_{011}.
  \nonumber \end{equation}
The three pairs of opposite faces therefore give the same compatibility condition
\begin{equation}
a+b+c-d=0.
  \nonumber \end{equation}
Choose
\begin{equation}
\begin{array}{c|cccccccc}
v&000&100&010&001&110&101&011&111\\ \hline
h&0&0&0&0&-1&0&0&-1.
\end{array}
  \nonumber \end{equation}
The lower faces of the lifted cube $\conv\{(v,h(v)):v\in V(P)\}$ project to the two triangular prisms
\begin{equation}
\begin{aligned}
P_A&=\conv\{000,001,010,011,110,111\},\\
P_B&=\conv\{000,001,100,101,110,111\},
\end{aligned}
  \nonumber \end{equation}
with common rectangle $R=\conv\{000,001,110,111\}$.

For the second refinement, after normalizing by affine functions on the two prisms, take
\begin{equation}
\begin{array}{c|rrrrrr}
v&000&001&010&011&110&111\\ \hline
h_A&0&0&0&-1&0&-2
\end{array}
\qquad
\begin{array}{c|rrrrrr}
v&000&001&100&101&110&111\\ \hline
h_B&0&0&0&-1&0&-2.
\end{array}
\nonumber
\end{equation}
These compatible heights give the regular subdivisions
\begin{equation}
\begin{gathered}
\conv\{000,010,110,111\},\quad
\conv\{000,010,011,111\},\quad
\conv\{000,001,011,111\},\\
\conv\{000,100,110,111\},\quad
\conv\{000,100,101,111\},\quad
\conv\{000,001,101,111\}.
\end{gathered}
  \nonumber \end{equation}
On the common rectangle $R$, both subdivisions use the diagonal $000$--$111$.
\end{example}

\begin{figure}[htbp]
\centering
\includegraphics[scale=0.45]{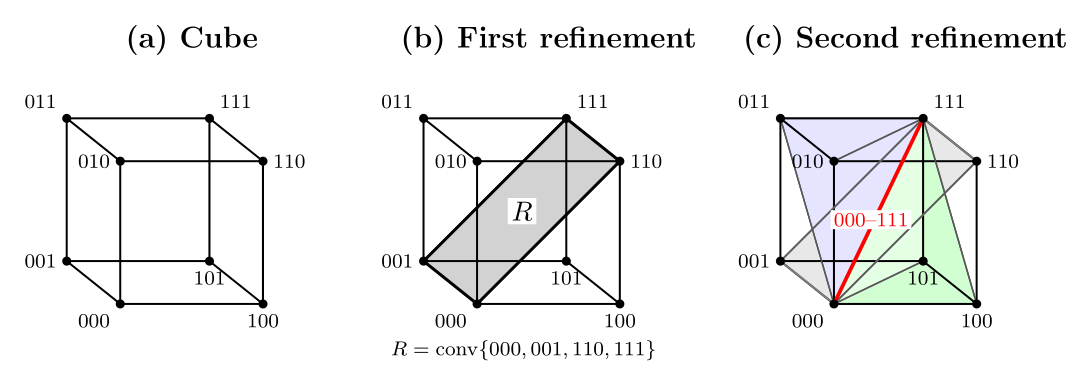}
\caption{
The cube example.
The first refinement introduces the interior face
$R=\operatorname{conv}\{000,001,110,111\}$, shown in gray.
The second refinement subdivides the two triangular prisms into tetrahedra.
The new triangular faces are shaded, and the common diagonal
$000$--$111$ of $R$ is shown in red.
}
\label{fig:cube-refinement}

\end{figure}

\FloatBarrier

\section[Proof of Theorem 1.1]{Proof of Theorem \ref{main}}\label{section:proof-main}
Let $P_i\subset A_i$, $i=1,\ldots,N$, be the lifts of the Epstein-Penner cells chosen in Section \ref{subsection:EP}. The face pairings induced by deck transformations are affine homeomorphisms, so the $P_i$ form an admissible polyhedral system. By Theorem \ref{theorem:affine-triangulation}, they admit compatible triangulations into tetrahedra using only their vertices.

By \eqref{equation:radial-convexity-prelim}, radial projection maps each tetrahedron to the convex hull of its ideal vertices. Let $v_1,\ldots,v_4\in P_i$ be the vertices of one of these tetrahedra. If
\begin{equation}
\sum_{k=1}^4 c_kv_k=0,
  \nonumber \end{equation}
then $\lambda_i(v_k)=1$ gives $\sum_kc_k=0$. Since the $v_k$ are affinely independent, $c_k=0$ for every $k$. Hence the vectors $v_k$ are linearly independent. Their positive rescalings $(1,\pi(v_k))$ are also linearly independent, so the four projected vertices are affinely independent in the Klein ball. Thus every projected tetrahedron is nondegenerate.

Extend the triangulations equivariantly to all lifts
of the Epstein--Penner cells. They agree on paired
faces and give a locally finite, $\Gamma$-invariant
face-to-face triangulation of $\mathbb H^3$.

There are finitely many tetrahedron orbits. Passing
to the quotient gives a finite ideal triangulation
of $M$. Since the tetrahedra geometrically subdivide
the Epstein--Penner cells, this triangulation
realizes the given complete hyperbolic structure.
If $M$ is orientable, the vertices of the tetrahedra
can be ordered so that all tetrahedra are positively
oriented.


\begin{corollary}\label{corollary:strict-angle}
Every finite-volume cusped hyperbolic $3$-manifold admits an ideal triangulation carrying a strict angle structure.
\end{corollary}

\begin{corollary}\label{corollary:positive-gluing}
Every orientable finite-volume cusped hyperbolic $3$-manifold has an ideal triangulation for which Thurston's gluing and completeness equations admit a solution with all shape parameters in the upper half-plane.
\end{corollary}

\begin{proof}
Take the shape parameters of the tetrahedra. The edge equations follow from the geometric face-to-face triangulation in $\H^3$, and the completeness equations follow from the induced Euclidean triangulations of sufficiently small horospherical cusp sections.
\end{proof}

\section{Projective face pairings and manifolds with geodesic boundary}\label{section:projective}
For manifolds with totally geodesic boundary, the face pairings of the associated affine fellows are generally projective rather than affine, because hyperbolic isometries act projectively in the Klein model. We therefore need a projective analogue of the affine triangulation theorem. We first record how regular subdivisions transform under a projective map.

Let $F,F'$ be convex polygons and let
\begin{equation}\label{equation:projective-map}
\phi(x)=\frac{Ax+b}{c\cdot x+d},\qquad \lambda(x)=c\cdot x+d>0
\end{equation}
be the restriction of a projective transformation from $F$ onto $F'$, defined on the whole closed polygon. For a height function $h$ on $V(F)$, define
\begin{equation}\label{equation:projective-transport}
(T_\phi h)(\phi(v))=\frac{h(v)}{\lambda(v)}.
\end{equation}

\begin{lemma}\label{lemma:projective-transport}
The map $T_\phi$ sends restrictions of affine functions on $F$ to restrictions of affine functions on $F'$, and
\begin{equation}
\phi(\Reg(F,h))=\Reg(F',T_\phi h).
  \nonumber \end{equation}
Hence it induces an isomorphism
\begin{equation}
\R^{V(F)}/\Aff(F)\longrightarrow\R^{V(F')}/\Aff(F').
  \nonumber \end{equation}
\end{lemma}

\begin{proof}
Use homogeneous coordinates $\widehat x=(x,1)$ and choose a homogeneous matrix $M$ representing $\phi$, so that
\begin{equation}
M\widehat x=\lambda(x)\widehat{\phi(x)}.
  \nonumber \end{equation}
Let $C$ be a cell of $\Reg(F,h)$, and choose an affine function $r(x)=a\widehat x$ such that $r(v)\leq h(v)$ for all $v\in V(F)$ and the contact vertices are precisely the vertices of $C$. Define $r'(y)=aM^{-1}\widehat y$. Then
\begin{equation}
r'(\phi(x))=\frac{r(x)}{\lambda(x)}.
  \nonumber \end{equation}
Together with \eqref{equation:projective-transport} and $\lambda>0$, this shows that the inequalities and equality vertices correspond exactly under $\phi$. Hence $\phi(C)$ is a cell of $\Reg(F',T_\phi h)$. Applying the same argument to $\phi^{-1}$ gives the reverse inclusion, and therefore the two regular subdivisions correspond. The same formula, without the inequalities, shows that $T_\phi$ sends restrictions of affine functions to restrictions of affine functions. This proves the lemma.
\end{proof}

\begin{theorem}\label{theorem:projective-triangulation}
Let $P_1,\ldots,P_N$ be three-dimensional convex polyhedra in affine charts, with their two-dimensional faces paired by restrictions of projective transformations defined on the closed faces. Then the $P_i$ can be subdivided into tetrahedra, without adding vertices, so that the triangulations agree on every pair of identified faces.
\end{theorem}

\begin{proof}
For a paired face $F^-_\alpha\xrightarrow{\phi_\alpha}F^+_\alpha$, impose
\begin{equation}
\left[h^+ -T_{\phi_\alpha}h^-\right]=0
\quad\text{in }\R^{V(F^+_\alpha)}/\Aff(F^+_\alpha).
  \nonumber \end{equation}
By Lemma \ref{lemma:projective-transport}, an $n$-gonal pair gives $n-3$ linear conditions, as in Lemma \ref{lemma:surplus}. Thus a nonzero compatible height class exists whenever some cell is not a tetrahedron. The induced subdivisions agree on paired faces, and new interior faces are paired by the identity. Iterating the construction and using the finiteness argument in Theorem \ref{theorem:affine-triangulation} gives the desired triangulations.
\end{proof}

\subsection{Affine fellows}
For a geodesic partially truncated polyhedron $P$ in the Klein ball, let $\widehat P$ denote its \emph{affine fellow} in the sense of Luo-Schleimer-Tillmann \cite{LST}. It is the convex affine polyhedron obtained by restoring the ideal and hyperideal vertices before truncation: ideal vertices lie on $\partial B^3$, while hyperideal vertices lie outside $\overline{B^3}$. The original polyhedron $P$ is recovered by truncating $\widehat P$ at its hyperideal vertices along the corresponding truncation planes. For the geometry and rigidity of hyperideal polyhedra, see \cite{BB}. For the geometric theory of partially
truncated triangulations and their consistency equations, see Frigerio-Petronio \cite{FP}.

The faces inherited from the untruncated polyhedron are called \emph{lateral faces}, while the new faces created by truncation are \emph{truncation faces}. Hyperbolic face pairings induce projective maps between the corresponding lateral faces of the affine fellows. Luo-Schleimer-Tillmann proved the local correspondence needed here.
\begin{figure}[htbp]
\centering
\includegraphics[scale=0.36]{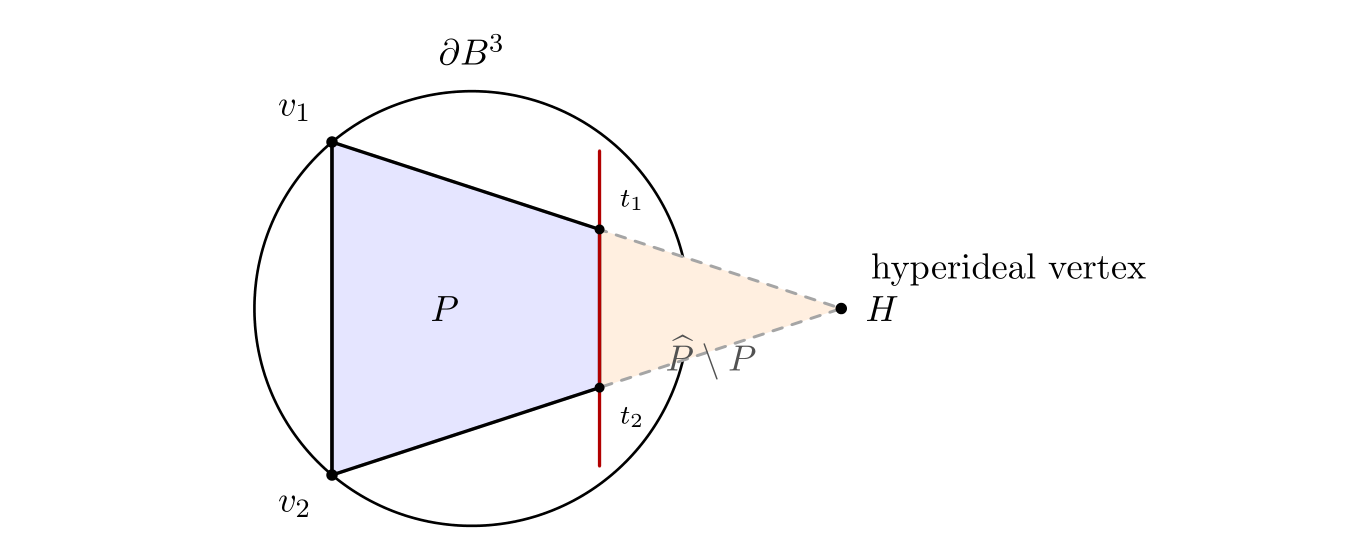}
\caption{
A geodesic partially truncated polyhedron and its affine fellow.
}
\label{fig:affine-fellow}
\end{figure}

\begin{theorem}[Luo-Schleimer-Tillmann \cite{LST}]\label{theorem:LST-fellow}
A subdivision of an affine fellow $\widehat P$ into straight tetrahedra using only its vertices determines a subdivision of the corresponding hyperbolic polyhedron into geodesic partially truncated tetrahedra.
\end{theorem}

The face pairings of the affine fellows extend the original hyperbolic face pairings and are projective homeomorphisms of the closed polygonal faces. In affine coordinates, their denominators are nowhere zero on these faces and can therefore be chosen positive. Thus they satisfy the hypotheses of Theorem \ref{theorem:projective-triangulation}.

\subsection[Proof of Theorem 1.2]{Proof of Theorem \ref{main:boundary}}
Kojima \cite{Kojima} gives a decomposition of $N$ into geodesic partially truncated polyhedra. In the setting of Theorem \ref{main:boundary}, the cells may be chosen so that each has at most one ideal vertex; see also \cite{LST,Frigerio}. Choose one affine fellow $\widehat P_i$ for every cell of the quotient decomposition. The hyperbolic face pairings induce projective transformations between the corresponding lateral faces of the fellows.

Apply Theorem \ref{theorem:projective-triangulation} to the fellows $\widehat P_i$. By Theorem \ref{theorem:LST-fellow}, their compatible triangulations determine a geometric partially truncated triangulation of $N$. No new vertices are introduced, and each Kojima cell has at most one ideal vertex; hence every tetrahedron has at most one ideal vertex. The triangulation is good and realizes the original hyperbolic structure.

\subsection[Proof of Theorem 1.3]{Proof of Theorem \ref{main:mixed}}
Ge-Jia-Zhang \cite{GJZdecomp} give a mixed polyhedral decomposition in which every cell is either ideal or has exactly one hyperideal vertex. Choose affine fellows for the finitely many quotient cells. Their paired lateral faces are related by projective transformations, so Theorem \ref{theorem:projective-triangulation} gives compatible triangulations of all fellows.

Theorem \ref{theorem:LST-fellow} converts these triangulations into partially truncated tetrahedra. A tetrahedron in an ideal cell is ideal. In a cell with one hyperideal vertex, a tetrahedron either avoids that vertex or contains it together with three ideal vertices. Hence every tetrahedron is ideal or of type $1$--$3$, and none is flat. The compatible triangulations of the lateral faces glue to a geometric triangulation realizing the original hyperbolic structure.

\section{Consequences}\label{section:consequences}
The triangulation of Theorem \ref{main} realizes the complete hyperbolic structure and determines the geometric data used below. Assume throughout this section that $M$ is orientable.

\subsection{Angle structures and the Casson-Rivin functional}
Let $\T$ be the geometric ideal triangulation given by Theorem \ref{main}. Suppose that $\T$ has $n$ tetrahedra and that $M$ has $k$ cusps. Denote by $\mathcal A(\T)$ the space of strict angle structures on $\T$. A point of $\mathcal A(\T)$ assigns three numbers in $(0,\pi)$ to the three pairs of opposite edges of every tetrahedron, with sum $\pi$ in each tetrahedron and sum $2\pi$ around every edge.

The dihedral angles give a point $\alpha_M\in\mathcal A(\T)$ by Corollary \ref{corollary:strict-angle}. Thus every cusped finite-volume hyperbolic $3$-manifold admits a strict-angle triangulation without a homological assumption. Hodgson-Rubinstein-Segerman \cite{HRS} obtained such triangulations under an additional homological hypothesis.

For $\theta\in[0,\pi]$, let
\begin{equation}
\Lambda(\theta)=-\int_0^\theta \log|2\sin t|\,dt
  \nonumber \end{equation}
be the Lobachevsky function. If $\alpha=(\alpha_1,\alpha_2,\alpha_3)$ are the angles of an ideal tetrahedron, its volume is
\begin{equation}
\Lambda(\alpha_1)+\Lambda(\alpha_2)+\Lambda(\alpha_3).
  \nonumber \end{equation}

Following the variational approach of Casson and Rivin \cite{FG,Riv}, summing over the tetrahedra defines the Casson-Rivin volume functional
\begin{equation}
\mathcal V:\overline{\mathcal A(\T)}\longrightarrow\R.
  \nonumber \end{equation}
It is continuous on the compact closure and strictly concave on $\mathcal A(\T)$; see \cite{FG}.

\begin{corollary}\label{corollary:casson-rivin}
For the triangulation $\T$ of Theorem \ref{main}, the angle polytope $\mathcal A(\T)$ is nonempty and has dimension $n+k$. Moreover
\begin{equation}\label{equation:volume-max}
\mathcal V(\alpha)\leq \Vol(M)
\qquad\text{for every }\alpha\in\overline{\mathcal A(\T)},
\end{equation}
with equality if and only if $\alpha=\alpha_M$. Hence $\alpha_M$ is the unique global maximizer of the Casson--Rivin functional on this triangulation.
\end{corollary}

\begin{proof}
The non-emptiness follows from Corollary \ref{corollary:strict-angle}, while the dimension formula $\dim\mathcal A(\T)=n+k$ is the standard rank computation for the angle equations; see \cite{FG}. Since $\alpha_M$ comes from the complete hyperbolic metric, the Casson--Rivin theorem identifies it with an interior critical point of $\mathcal V$; see \cite{FG}. The volume at this point is the sum of the volumes of the geometric tetrahedra, hence
\begin{equation}
\mathcal V(\alpha_M)=\Vol(M).
  \nonumber \end{equation}
The strict concavity and the volume-maximization theorem of Casson and Rivin then give \eqref{equation:volume-max}, including the equality statement; see \cite{FG,Riv}.
\end{proof}
\begin{remark}
Corollary \ref{corollary:casson-rivin} applies to the triangulation constructed in Theorem \ref{main}. It does not establish Casson's conjectured volume inequality for arbitrary angled ideal triangulations of $M$.
\end{remark}

\subsection{Shape parameters, deformations and the Bloch invariant}
Let $z_1,\ldots,z_n\in\mathbb C$ be the shapes of the tetrahedra of $\T$, with the orientation chosen as in Theorem \ref{main}. Then
\begin{equation}
\operatorname{Im}z_j>0
\qquad\text{and}\qquad
\Vol(M)=\sum_{j=1}^n D(z_j),
  \nonumber \end{equation}
where $D$ is the Bloch--Wigner dilogarithm. Corollary \ref{corollary:positive-gluing} says in addition that these shapes satisfy both the edge equations and the peripheral completeness equations.

Let $\mathcal G(\T)$ denote the variety defined near $z=(z_1,\ldots,z_n)$ by the edge equations alone. The local deformation theory of Thurston and Neumann--Zagier applies at this complete positive solution.

\begin{corollary}\label{corollary:NZ}
If $M$ has $k$ cusps, the complete shape vector $z$ is a nonsingular point of $\mathcal G(\T)$, and $\mathcal G(\T)$ has complex dimension $k$ near $z$. The logarithmic holonomies of one nontrivial peripheral curve on each cusp give local holomorphic coordinates. In particular, the usual hyperbolic Dehn filling deformation theory can be carried out starting from a solution for which every tetrahedron has positive imaginary part.
\end{corollary}

\begin{proof}
By Corollary \ref{corollary:positive-gluing}, the triangulation admits a complete positively oriented
solution of the gluing equations. The smoothness and dimension assertions, as well as the local parametrization by peripheral holonomies, follow
from Choi's deformation theorem \cite{Choi} (or see Thurston \cite{Thurston} and Neumann-Zagier \cite{NZ}).
\end{proof}

The same shape parameters represent the Bloch invariant. Neumann-Yang \cite{NY} associate to a finite-volume hyperbolic $3$-manifold its Bloch invariant $\beta(M)\in\mathcal B(\mathbb C)$, represented by the simplex parameters of a degree-one ideal triangulation, with signs determined by orientation. Since every tetrahedron of $\T$ is positively oriented, the triangulation of Theorem \ref{main} gives the particularly simple representative
\begin{equation}\label{equation:bloch-positive}
\beta(M)=\sum_{j=1}^n [z_j],
\qquad \operatorname{Im}z_j>0.
\end{equation}
The Bloch-Wigner regulator of \eqref{equation:bloch-positive} is $\Vol(M)$. After compatible flattenings are chosen, the extended Bloch-group formalism of Neumann \cite{NeumannExtended} gives the corresponding complex-volume and Chern-Simons classes from the same triangulation.

\subsection{Efficiency, the $3$D index and quantum gluing data}
A strict angle structure implies strong restrictions on normal surfaces and also supplies an index structure.

\begin{corollary}\label{corollary:efficiency-index}
The triangulation $\T$ of Theorem \ref{main} is strongly $1$-efficient. In particular, it has no spun normal surface of Euler characteristic zero. Moreover $\T$ admits an index structure and its $3$D index is well-defined.
\end{corollary}

\begin{proof}
By Corollary \ref{corollary:strict-angle}, $\T$ carries an angle structure. Kang--Rubinstein \cite[Theorem 2.5]{KR} show that an ideal triangulation carrying an angle structure is strongly $1$-efficient, and \cite[Theorem 2.8]{KR} excludes spun normal surfaces of Euler characteristic zero. A strict angle structure is, in particular, an index structure in the sense of Garoufalidis \cite{Garoufalidis}; equivalently, one may use the characterization of index structures by $1$-efficiency in \cite{GHRS}. Hence the $3$D index of $\T$ is defined.
\end{proof}

The Neumann-Zagier data of $\T$ also enter quantum constructions based on the gluing equations. Dimofte-Garoufalidis \cite{DG} associate perturbative invariants to a Neumann-Zagier datum coming from a regular ideal triangulation; Theorem \ref{main} gives such a datum for the complete representation with all shapes in the upper half-plane. Garoufalidis-Kashaev \cite{GK} associate a meromorphic state integral to an ideal triangulation, and when a strict angle structure exists its Laurent expansion has coefficients given by the $3$D index. Thus the gluing-equation, Casson-Rivin and $3$D-index constructions can all be carried out on the same geometric triangulation.

\subsection{Manifolds with geodesic boundary}
The dihedral angles of the triangulations in Theorems \ref{main:boundary} and \ref{main:mixed} satisfy the corresponding generalized angle equations.

\begin{corollary}\label{corollary:generalized-angle}
The triangulations in Theorems \ref{main:boundary} and \ref{main:mixed} carry positive generalized angle structures given by their dihedral angles. At every ideal vertex the incident angles sum to $\pi$, at every hyperideal vertex they sum to a number strictly smaller than $\pi$, and around every internal edge the angles sum to $2\pi$.
\end{corollary}

\begin{proof}
The link of an ideal vertex is Euclidean and the link of a hyperideal vertex is hyperbolic. The vertex relations follow from the angle sums in these links. Around an internal edge, a small transverse disk is divided into sectors whose angles are the dihedral angles of the incident tetrahedra, so their sum is $2\pi$.
\end{proof}

For the mixed manifolds of Ge-Jia-Zhang these angles come directly from a geometric subdivision and no flat tetrahedra occur. The angle-structure construction in \cite{GJZangle} uses additional topological hypotheses.

The good triangulations of Theorem \ref{main:boundary} also remove the existence assumption from Frigerio's triangulation-based deformation theory.

\begin{corollary}\label{corollary:frigerio-deformation}
Let $N$ satisfy the hypotheses of Theorem \ref{main:boundary} and suppose that $N$ has $k$ toric cusps. There is a good geometric partially truncated triangulation of $N$ for which the space of solutions to the consistency equations is a smooth real $2k$-manifold near the point representing the complete structure.
\end{corollary}

\begin{proof}
Apply Frigerio's local deformation theorem \cite{Frigerio} to the good geometric triangulation of Theorem \ref{main:boundary}.
\end{proof}

Related variational results are known for ideal and hyperideal polyhedral metrics. Luo-Yang \cite{LuoYang} developed volume and rigidity methods for ideal and hyperideal polyhedral metrics, and Feng-Ge-Liu \cite{KGL} proved a Casson-Rivin type volume-optimization theorem for triangulations made of $1$-$3$ tetrahedra, namely tetrahedra with one hyperideal and three ideal vertices. The mixed triangulations of Theorem \ref{main:mixed} may contain both ideal and $1$-$3$ tetrahedra, so those results do not apply directly to the whole mixed triangulation.

\noindent Huabin Ge, hbge@ruc.edu.cn\\[2pt]
\emph{School of Mathematics, Renmin University of China, Beijing 100872, P. R. China}\\[2pt]

\end{document}